\documentclass[11pt,a4paper]{article} 
\usepackage{fullpage} 
\usepackage[latin1]{inputenc}
\usepackage[english]{babel}
\usepackage[T1]{fontenc}
\usepackage{lmodern} 
\usepackage{graphicx}
\usepackage{wrapfig}
\usepackage{subfigure}
\usepackage{captcont}
\usepackage{enumitem}
\usepackage{datetime}
\usepackage{longtable}
\usepackage{amsmath,amssymb}
\usepackage{algorithm}
\usepackage{algorithmic}
\usepackage{lineno}

\DeclareMathOperator{\conv}{conv}

\newcommand{\titel}{Which point sets admit a $k$-angulation?} 

\usepackage[usenames]{color} 
\definecolor{hellblau}{rgb}{0.2,0.4,1} 
\definecolor{dunkelblau}{rgb}{0,0,0.8}
\definecolor{dunkelgruen}{rgb}{0,0.5,0}
\usepackage[
	pdftex,
	colorlinks,
	linkcolor=dunkelblau,
	urlcolor=dunkelblau,
	citecolor=dunkelgruen,
	bookmarks=true,
	linktocpage=true,
	pdftitle={\titel},
	pdfauthor={},
	pdfsubject={},
	pdfkeywords={}%
]{hyperref} 
\usepackage{amsmath} 
\usepackage{amsthm} 
\usepackage{amsfonts} 
\theoremstyle{plain} 
	\newtheorem{satz}{Satz}[] 
	\newtheorem{theorem}[satz]{Theorem}
	\newtheorem{lemma}[satz]{Lemma}
	
	\newtheorem{proposition}[satz]{Proposition}
	\newtheorem{conjecture}[satz]{Conjecture}
\theoremstyle{remark} 
	\newtheorem*{claim}{Claim 1}
\theoremstyle{definition} 

	\newtheorem{corollary}[satz]{Corollary}

\title%
{\textbf{\titel}%
    {\footnotetext%
        {This research was supported by the DAAD and the Go8 within the Australia--Germany Joint Research Co-operation Scheme 2011 as part of the project \emph{Problems in geometric graph theory} (Kennz. 50753217).%
        }\renewcommand\footnotemark{}%
    }%
}

\author{Michael S. Payne\footnote{Department of Mathematics and Statistics, The University of Melbourne, m.payne3@pgrad.unimelb.edu.au.} 
\and Jens M. Schmidt\footnote{Max Planck Institute for Informatics, Saarbr\"ucken, jens.schmidt@mpi-inf.mpg.de.} 
\and David R. Wood\footnote{School of Mathematical Sciences, Monash University, david.wood@monash.edu. Supported by a QEII Fellowship and a Discovery Project from the Australian Research Council.}}

\begin{document}
\maketitle

\begin{abstract}
For $k \geq 3$, a $k$-angulation is a $2$-connected plane graph in which every internal face is a $k$-gon.
We say that a point set $P$ \emph{admits} a plane graph $G$ if there is a straight-line drawing of $G$ that maps $V(G)$ onto $P$ and has the same facial cycles and outer face as $G$.
We investigate the conditions under which a point set $P$ admits a $k$-angulation and find that, for sets containing at least $2k^2$ points,  the only obstructions are those that follow from Euler's formula.
\end{abstract}


\section{Introduction}

A point set $P$ of size $n$ \emph{admits} an $n$-vertex plane graph $G$ if there is a straight-line drawing of $G$ on $P$ that has the same facial cycles and outer face as $G$.
Consider the following general problems. Firstly, given a class $\mathcal{C}$ of plane graphs and a set $P$ of points in the plane, does $P$ admit some graph in $\mathcal{C}$? 
Secondly, classify all (general position) point sets that admit at least one graph in $\mathcal{C}$.
In the present paper we consider the case where $\mathcal{C}$ is the set of $k$-angulations. 

For $k \geq 3$, a \textit{$k$-angulation} is a $2$-connected plane graph in which every internal face is a $k$-gon. 
$2$-connectedness is a natural property to require because it is equivalent to all the faces being bounded by simple cycles (see for example~\cite{Thomassen1981}).
A point in $P$ is an \emph{interior point of $P$} if it is not on the boundary of the convex hull of $P$.
Our main result is that 
all point sets in general position with at least $2k^2$ points admit a $k$-angulation unless they have too few interior points, which is a necessary condition that follows from Euler's formula. More precisely, we prove the following theorem.
We abbreviate $a\equiv b \bmod c$ as $a\equiv_c b$.

\begin{theorem}\label{mainthm}
Let $n \geq 2k^2$ and $j \equiv_{k-2} k-n$ with $0 \leq j \leq k-3$.
A set of $n$ points in general position in the plane
admits a $k$-angulation if and only if it has at least $j$ interior points.
\end{theorem}

In related work, Bose and Toussaint~\cite{Bose1995,Bose1997,Toussaint1995} showed that $P$ admits a quadrangulation \textit{with $\conv(P)$ as the outer face} if and only if the number of points on the boundary of $\conv(P)$ is even.
Moreover, they showed that all point sets admit such $k$-angulations when the addition of at most $k-3$ extra points to $P$ is allowed. 
We are not aware of any results for $k$-angulations in which the outer face is allowed to be non-convex. In our work we do not allow the addition of extra points.
In a similar vein, Gritzmann et al.~\cite{gritzmohar} showed that every point set $P$ admits every $|P|$-vertex outerplanar graph 
(Castañeda and Urrutia~\cite{CastUrrutia} rediscovered this result).
Bose \cite{Bose02} gave near optimal algorithms for achieving this.


Dey et al.\ \cite{Dey1997} characterise the point sets with three vertices on the convex hull that admit a $4$- or $5$-connected triangulation. They also note that the only point sets $P$ that do not admit a $3$-connected graph (if $|P| > 3$) are the ones in convex position (the same characterization holds for $3$-edge-connected graphs). Garc{\'\i}a et al.~\cite{Garcia2009} determine the minimum number of edges in a $3$-connected graph on $P$ and characterise the point sets $P$ that admit $3$-regular $3$-connected graphs.
Schmidt and Valtr~\cite{SchmidtValtr} characterise 
the point sets that admit $3$-regular graphs and obtain a polynomial time algorithm that constructs such a $3$-regular graph if one exists. 

Regarding the complexity of embedding a given graph $G$ on a given point set $P$, Cabello~\cite{Cabello} showed that deciding whether $P$ admits a straight line drawing of $G$ is NP-complete. 
His construction used $2$-connected graphs which may have many different drawings. 
Biedl and Vatshelle~\cite{BiedlVatshelle} proved NP-completeness in the $3$-connected case. 
For more on drawing planar graphs on a given point set see~\cite{more3, more1, more4, more2}.


\section{Preliminaries}

So as to avoid any potential confusion, we begin by formally defining some terms used in the introduction.
A \emph{planar drawing} of a graph $G$ is a function $\varphi$ that maps each vertex $v \in V(G)$ to a unique point in the plane and each edge $ab \in E(G)$ to a simple curve in the plane that connects $\varphi(a)$ to $\varphi(b)$, such that the curves only intersect each other and $\varphi(V(G))$ at their endpoints.
A \emph{straight-line drawing} is a planar drawing that maps each edge to a straight line segment.
A planar drawing of $G$ determines a set of closed facial walks including one distinguished as the outer face. 
Two planar drawings of $G$ are \emph{equivalent} if they determine the same set of facial walks and they have the same outer face.
Formally, a \emph{plane graph} is a planar graph $G$ together with an equivalence class $\tilde\varphi$ of planar drawings under this equivalence relation. 
Informally, a plane graph is an embedded planar graph with a nominated outer face.
An \emph{internal vertex} of a plane graph is a vertex not on the outer face.
We say that a point set $P$ \emph{admits} a plane graph $(G, \tilde\varphi)$ if there is a straight-line drawing of $G$ that maps $V(G)$ onto $P$ and is a member of $\tilde\varphi$.
Throughout this paper, let $P$ be a set of at least $3$ points in general position in the plane; that is, no three points of $P$ are collinear.
One direction of Theorem~\ref{mainthm} is straightforward to prove. In fact we can strengthen it to apply for all $n$.

\begin{proposition}\label{lem:interiorpoints}
Suppose a point set $P$ with $n=|P| \geq k$ admits a $k$-angulation $G$. 
Let $j \equiv_{k-2} k-n$ with $0 \leq j \leq k-3$. 
Then $P$ contains at least $j$ interior points.
\end{proposition}

\begin{proof}
Since $G$ is $2$-connected the boundary of the outer face is a cycle $C$; say it has $r$ vertices.
$P$ has at least $n-r$ interior points since every internal vertex must be an interior point. 
Let $e$ be the number of edges and $f$ be the number of faces in $G$. Thus $2e = k(f-1)+r$. Euler's formula states that $n-e+f = 2$. 
Thus, $n-r = n- 2e +k(f-1) = n- 2(n+f-2) + k(f-1) = k - n + (f-2)(k-2)$. 
Hence $n-r \equiv_{k-2} k-n$.
Since $n-r$ is non-negative, $P$ contains at least $j$ interior points.
\end{proof}

Note that in the case $k=3$ this proposition gives no information about the point set $P$ because every $P$ admits a triangulation. In the case $k=4$, if $P$ admits a $k$-angulation then if $n$ is even there may be no interior points, but if $n$ is odd there is at least one.

\medskip

The proof of the other direction of Theorem~\ref{mainthm} takes a lot more work, beginning with some more definitions. 
The \emph{dual graph} of a $2$-connected plane graph $G$ is the multigraph with the faces of $G$ as vertex set (including the outer face), and an edge between any two faces for each edge of $G$ that they share. 
Thus, every edge in the primal graph $G$ has a corresponding dual edge in the dual graph, and vice versa.
Also, the dual of a $2$-connected plane graph has no loops.
The \emph{weak dual graph} of $G$ is obtained from the dual graph by deleting the outer face vertex. 

Our general approach for finding $k$-angulations is to construct a triangulation $G$ of $P$ that can be made into a $k$-angulation by removing some edges.
Thus each internal face of the $k$-angulation will be the union of some triangular faces of $G$. 
To determine the edges to be removed, it is most convenient to consider the weak dual $G^*$.
Note that the weak dual of a triangulation does not have parallel edges, since two triangular faces share at most one edge.
We aim to partition $G^*$ into disjoint connected subgraphs that cover all the vertices of $G^*$.
We call these connected subgraphs of $G^*$ \emph{blocks}\footnote{Not to be confused with the common usage of the term `block' meaning a maximal $2$-connected subgraph.} and such a partition a \emph{block partition} of $G^*$.
For any block partition of $G^*$, the \emph{corresponding subgraph} of $G$ is the subgraph formed by removing the edges that correspond (by duality) to the edges of each block.
The \emph{order} of a block, and the order of a graph in general, is the number of vertices it has.


\begin{proposition}\label{2con} Let $G$ be a triangulation of $P$ and suppose that the weak dual $G^*$ is partitioned into blocks of order $k-2$. 
Let $G'$ be the subgraph of $G$ corresponding to this block partition.
If $G'$ is $2$-connected, then it is a $k$-angulation of $P$.
\end{proposition}

\begin{proof}
Since $G'$ is $2$-connected, all its faces are bounded by simple cycles. 
It remains to show that the internal faces are bounded by $k$-cycles.
Each internal face of $G'$ corresponds to a block with $k-2$ vertices in $G^*$.
If a block contains a cycle, then the boundary of the corresponding face is not connected.
Therefore, every block is a tree.
Thus the number of edges in every block is $k-3$.
This implies that the number of edges in the boundary of every internal face is $3(k-2) - 2(k-3) = k$.
\end{proof}

Except when $P$ is in convex position, we begin with the \emph{wheel triangulation}, which is constructed as follows.
Choose an interior point $z$ of $P$, insert an edge from $z$ to each point in $P \setminus \{z\}$, then add the cycle that passes through $P\setminus \{z\}$ in radial order about $z$. 
This gives a $3$-connected triangulation of $P$ whose weak dual is a cycle $Z$ (see Figure~\ref{fig:Bad}).
We then add triangles one by one to the outside of the wheel triangulation. In the weak dual, this corresponds to pasting binary trees onto $Z$ by their leaves. 
Next some flipping operations in the primal graph (called building pontoons) are used to connect these trees by paths separate from the (now modified) inner cycle.
The result is the triangulation which is to be partitioned into blocks to yield a $k$-angulation.

The points of $P$ are classified according to their role in the wheel triangulation. Let $C$ be the outer cycle of the wheel triangulation. Then exactly $n-1$ vertices are contained in $C$ and we call the only vertex $z$ that is not in $C$ the \emph{central vertex}. A vertex $v$ in $C$ is called \emph{reflex} if its exterior angle is less than $\pi$. Otherwise, its exterior angle is greater than $\pi$ and $v$ is called \emph{non-reflex}. A non-empty path in $C$ whose vertices are all non-reflex is called a \emph{convex path}. 

For a non-empty path $A$ in $C$, 
 let $p(A)$ be the vertex in $C$ that is the clockwise \emph{predecessor} of $A$ and let $s(A)$ be the vertex in $C$ that is the clockwise \emph{successor} of $A$. Let the \emph{closure} of $A$, denoted $\bar{A}$, be the path in $C$ from $p(A)$ through $A$ to $s(A)$. If $A$ has order $n-2$ or $n-1$ then $\bar{A}$ is the whole cycle $C$. 
A convex path $A \subseteq C$ such that $z \in \conv(\bar{A})$ is called \emph{bad} (see Figure~\ref{fig:Bad}); otherwise it is called \emph{good} (see Figure~\ref{fig:Pontoon1}).


\begin{figure}
	\centering
	\includegraphics[scale=0.5]{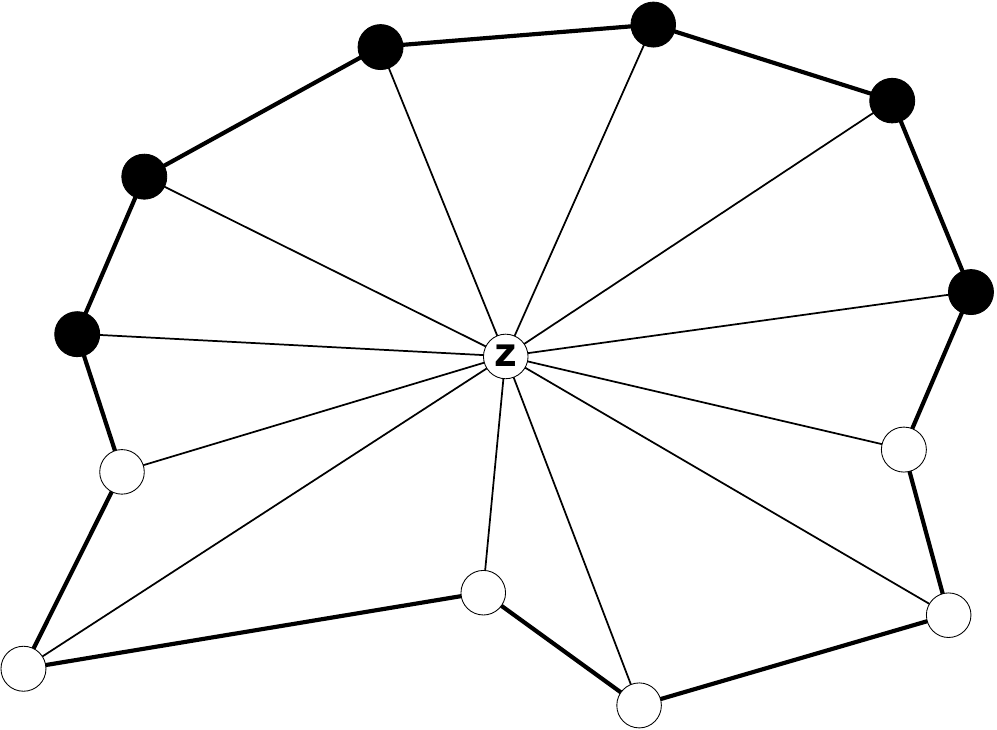}
	\caption{A wheel triangulation. The black vertices induce a maximal bad convex path.}
	\label{fig:Bad}
\end{figure}

\begin{lemma}
There is at most one maximal bad convex path in $C$.
\end{lemma}
\begin{proof}
Let $B$ be an inclusion maximal bad convex path in $C$. By definition, $z$ is contained strictly inside the convex hull of $\bar{B}$. Thus, $\bar{B}$ covers an angle of more than $\pi$ at $z$. Since $B$ is maximal convex, any other maximal bad convex path must be disjoint from $\bar{B}$, but this would exceed the total angle $2\pi$ at $z$.
\end{proof}

Note that for a good convex path $A$, the vertices $p(A)$ and $s(A)$ must be distinct, and $C\setminus \bar{A}$ must contain at least one vertex since $z \not\in \conv(\bar{A})$. Moreover, the points in $\bar{A} \cup \{z\}$ are in convex position (see Figure~\ref{fig:Pontoon1}). 
For any good convex path $A$, we now define a special retriangulation of $\conv(\bar{A} \cup \{z\})$ called a \emph{pontoon} over $A$.
First delete the edges of the initial wheel triangulation that lie in the interior of $\conv(\bar{A} \cup \{z\})$. 
Then add all edges from $s(A)$ to other vertices in $\bar{A} \cup \{z\}$ (see Figure~\ref{fig:Pontoon2}). 
The \emph{order} 
of a pontoon is the order of the path over which it is built.
Note that a pontoon cannot be built over a bad convex path.
The reason for the name of this construction will become apparent in the proof of Theorem~\ref{mainthm}, where it is used to connect triangles added outside of $C$ by a structure similar to a bridge or pontoon in the weak dual.

\begin{figure}
	\centering
	\subfigure[A good convex path $A$, depicted with black vertices. The grey triangles will be added to the triangulation to find a $k$-angulation later.]{
			\includegraphics[scale=0.6]{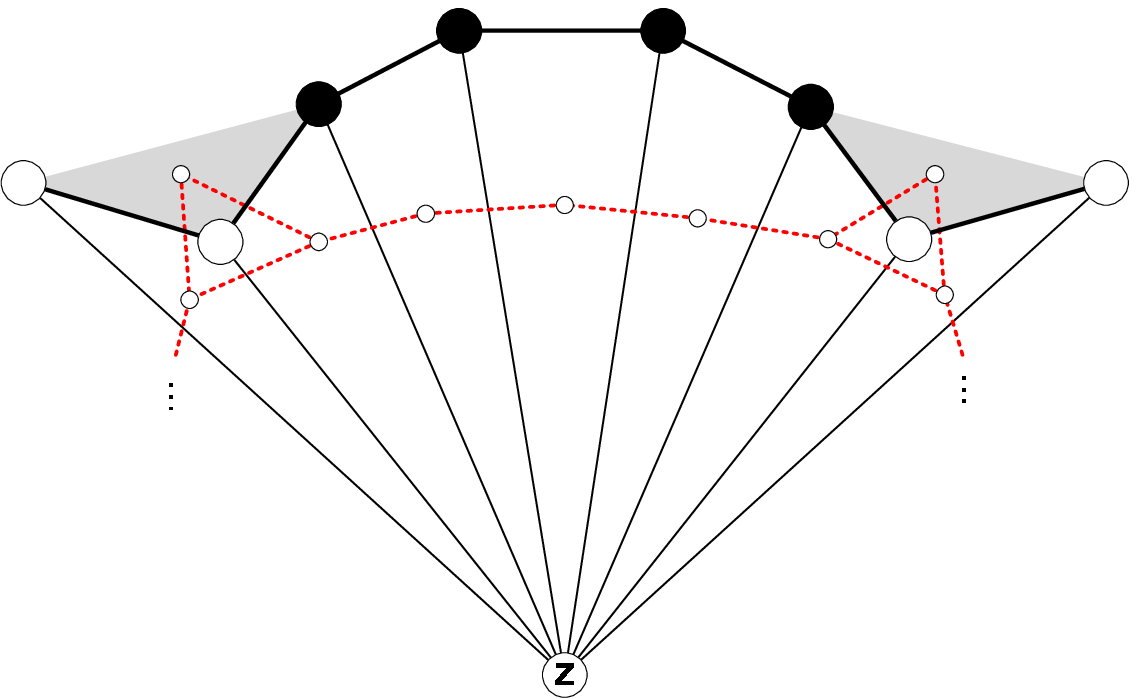}
			\label{fig:Pontoon1}
	}
	\quad\quad
	\subfigure[A pontoon is built over $A$ by retriangulating $\conv(\bar{A} \cup \{z\})$.  The weak dual is depicted with dotted edges.]{
			\includegraphics[scale=0.6]{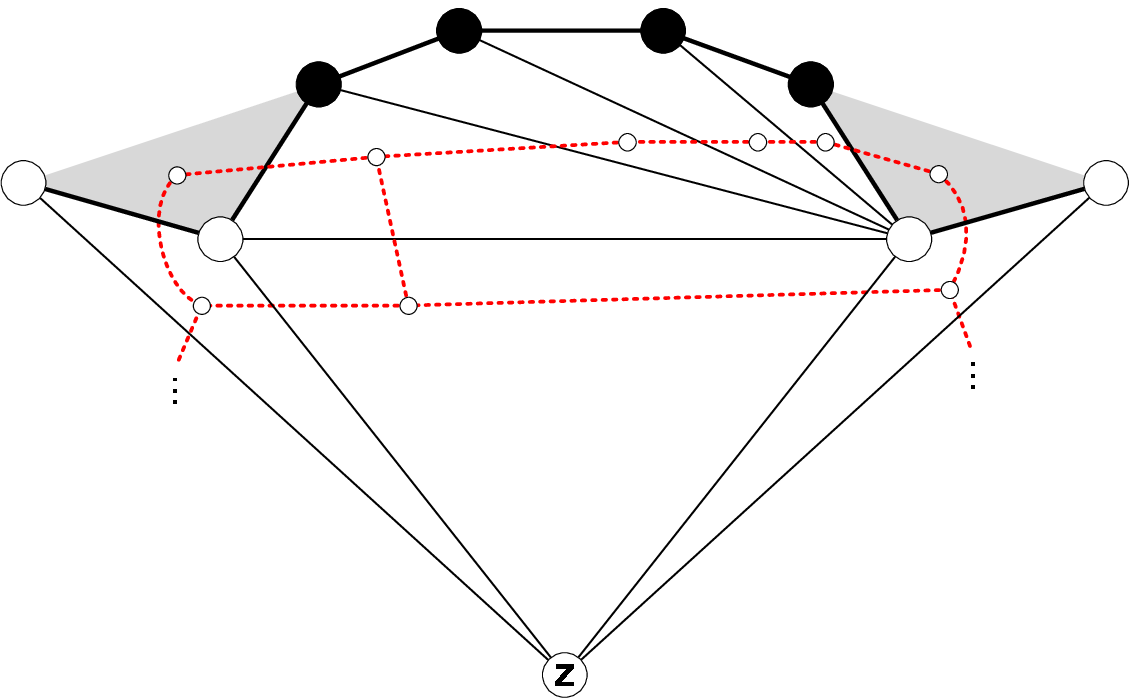}
			\label{fig:Pontoon2}
	}
	\caption{Building Pontoons}
	\label{fig:Pontoons}
\end{figure}


\section{Proof}

We are now ready to begin the proof of the other direction of Theorem~\ref{mainthm}, namely:


\begin{theorem}
Let $n \geq 2k^2$ and $j \equiv_{k-2} k-n$ with $0 \leq j \leq k-3$.
Every set $P$ of $n$ points in general position in the plane with at least $j$ interior points admits a $k$-angulation.
\end{theorem}

\begin{proof}

If $j=0$, then $n \equiv_{k-2} k$. 
There exists an outerplanar $k$-angulation with $f$ faces for every $f$, and the number of vertices is $k + (f-2)(k-2) \equiv_{k-2} k$. 
Therefore, using the result of Gritzmann et al.~\cite{gritzmohar} mentioned above, a $k$-angulation can be drawn on $P$.

If $j > 0$ then there is at least one interior point. 
Select an arbitrary interior point $z$ and construct the wheel triangulation with $z$ as the central vertex.
If $j=1$ then the number of triangles in the wheel triangulation is $n-1 \equiv_{k-2} k-j-1 = k-2$. 
The weak dual is a cycle on $n-1 = c(k-2)$ vertices, where $c\geq 2$ since $n \geq 2k^2$. 
This cycle can be partitioned into $c$ blocks, each of which is a path of order $k-2$.
The subgraph of the wheel triangulation corresponding to this block partition is a subdivision of a wheel graph with $c$ spokes, and is therefore $2$-connected.
Hence, by Proposition~\ref{2con}, it is a $k$-angulation.

Now assume $2 \leq j \leq k-3$. 
%
We start by using Algorithm~\ref{alg:add} to add $m:=j-1$ triangles to the wheel triangulation. 
Note that $m$ is the number of triangles we must add to the $n-1$ triangles of the wheel triangulation to make the total number of triangles a multiple of $k-2$. 
As triangles are added, the current outer cycle is denoted $C'$. 
Reflex and non-reflex vertices on $C'$ are defined analogously to those on $C$.
Each time a triangle is added, a reflex vertex of $C'$ becomes an internal vertex. 
In this way, we may identify each triangle with a unique vertex, and each triangulation (which extends the wheel triangulation) with a sequence of vertices. 

Algorithm~\ref{alg:add} is initialised at a non-reflex vertex on $C':= C$. 
Moreover, if there is a maximal bad convex path $B$, it is initialised at a non-reflex vertex inside $B$. 
It then proceeds clockwise around $C'$ (see Figure~\ref{fig:AddingTriangles}). 
Each time it encounters a reflex vertex of $C'$, it 
adds a triangle there, proceeds to the next vertex clockwise around $C'$, and then updates $C'$.
Algorithm~\ref{alg:add} stops when it has added $m$ triangles. The output is the sequence $S:=(s_1,\dots,s_m)$ of vertices in $C$ where triangles were added.

Algorithm~\ref{alg:add} may complete many laps around $C'$.
If at some time no more triangles can be added, then $C'$ must be a convex polygon. 
In this case, all interior points of $P$ have been removed from $C'$, so at least $m$ triangles have been added.
Therefore Algorithm~\ref{alg:add} always terminates.

\begin{algorithm}
\caption{Adds $m$ triangles to a wheel triangulation $G$ of $P$}\label{alg:add}
\begin{algorithmic}[1]
\STATE $t \gets 1$; $C' \gets C$

\IF {there exists a bad convex path $B$} \STATE select a vertex $v \in B$ 
	\ELSE \STATE $v \gets$ any non-reflex vertex 
\ENDIF

\WHILE {$t \leq m$} 
   \STATE let $u$ be the vertex anticlockwise from $v$ on $C'$
   \STATE let $w$ be the vertex clockwise from $v$ on $C'$
	\IF {$v$ is reflex in $C'$}
		\STATE add $uw$ to $G$
		\STATE $C' \gets (C' \setminus \{uv, vw\} ) \cup  \{ uw \}$
		\STATE $s_t \gets v$
		\STATE $t \gets t+1$
	\ENDIF
   \STATE $v \gets w$
\ENDWHILE
\RETURN $(s_1,\dots,s_m)$
\end{algorithmic}
\end{algorithm}

\begin{figure}
	\centering
	\includegraphics[scale=0.65]{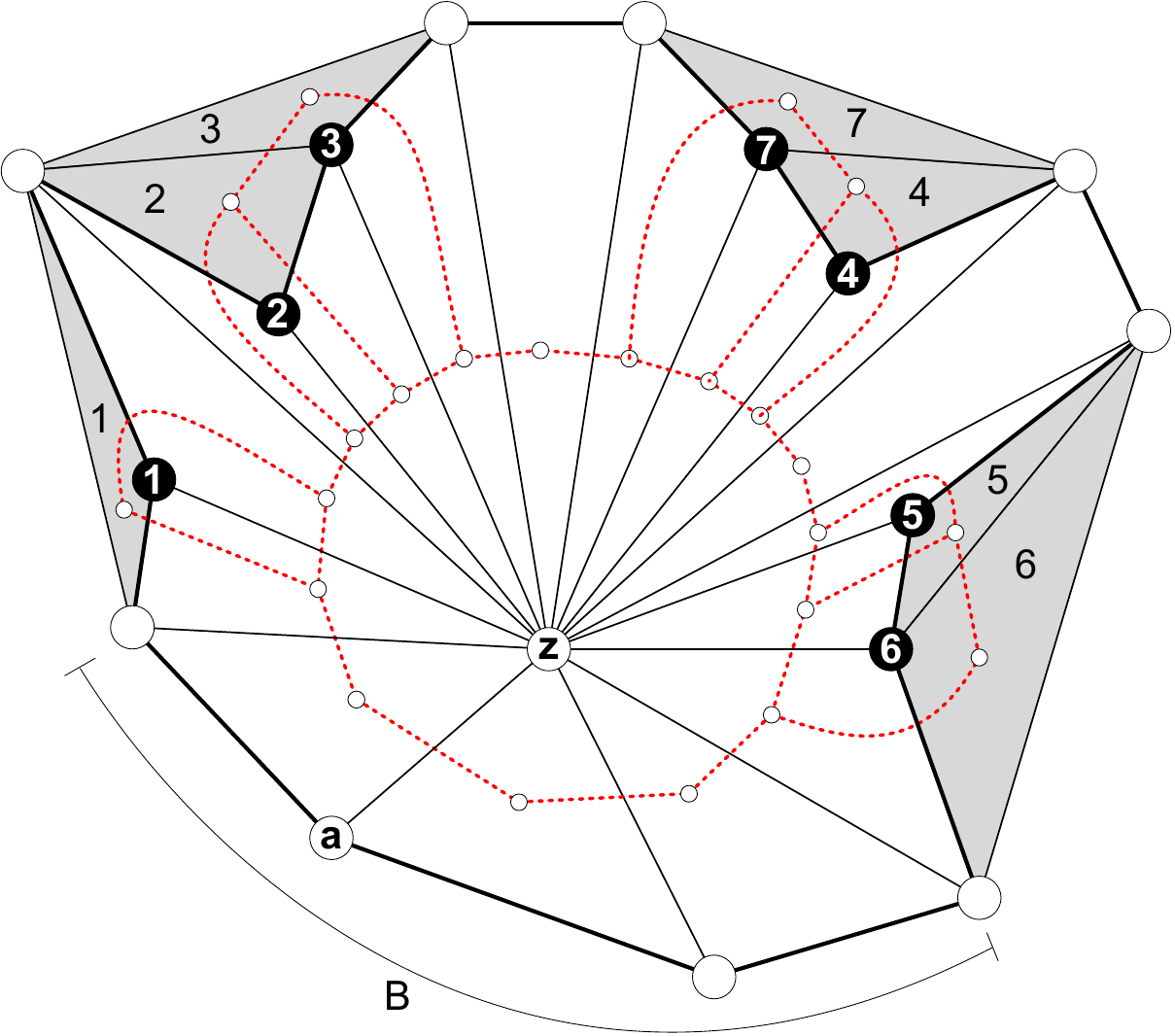}
	\caption{Algorithm~\ref{alg:add} starts at $a$ inside the bad path $B$ and adds the grey triangles to the wheel triangulation in the given order. The black vertices depict $S$. Dashed edges depict the weak dual graph with four trees pasted onto $Z$.}
	\label{fig:AddingTriangles}
\end{figure}

Let $G_t$ be the triangulation and $G^*_t$ the weak dual after Algorithm~\ref{alg:add} has added $t$ triangles.

\begin{claim} $G^*_t - E(Z)$ is a complete binary forest with all leaves and isolated vertices in $V(Z)$. Furthermore, for each component tree $T$ in this forest, the leaves of $T$ are consecutive in $Z$.
\end{claim}

\begin{proof} For the first claim do induction on $t$. 
For $t=0$, $G^*_0 - E(Z) = Z- E(Z)$ is a set of isolated vertices.
Now suppose $G^*_t - E(Z)$ is a complete binary forest with all leaves and isolated vertices in $V(Z)$.
Complete binary trees have a unique root which has degree 2 or 0, all other vertices have degree 3 or 1.
In the primal graph, each new triangle is added at a reflex vertex, using two existing edges and adding one new edge.
This corresponds to connecting a new degree 2 vertex to two vertices $v$ and $w$ of degree 2 in $G^*_t$. (There are no degree $1$ vertices in $G^*_t$, and degree $3$ vertices correspond to triangles that do not share an edge with $C'$). 
Removing the edges of $Z$ reduces the degree of each vertex by 0 or 2.
Hence $v$ and $w$ have degree 0 or 2 in $G^*_t - E(Z)$, and are therefore roots. 
Since each tree has one root, each vertex addition joins two separate trees, giving them a new root.
A new leaf is created only if one of these trees was an isolated vertex.
Thus $G^*_{t+1} - E(Z)$ is also a complete binary forest with leaves and isolated vertices in $V(Z)$.

To see that the leaves of a given tree $T$ are consecutive in $Z$, suppose they are not. 
Then there is a vertex $v$ of $Z$ that is not a leaf of $T$ between two leaves $l_1$ and $l_2$ of $T$.
There is a cycle $\tilde{Z}$ in $G^*_t$ that passes from $l_1$ to $l_2$ through $T$ and back to $l_1$ through $Z$ and contains $v$ in its interior. 
Since the trees cover all vertices of $G^*_t$, $v$ is in a different tree $T'$.
The triangle corresponding to the root of $T'$ has an edge in the outer face, but $T'$ is separated from the outer face by $\tilde{Z}$.
\end{proof}

Recall that $S$ is the set of vertices in $C$ on which Algorithm~\ref{alg:add} added triangles. Consider the collection of paths induced by $S$ in $C$. 
Each such path contains a reflex vertex of $C$, otherwise it would be a convex path and no triangle could be added on it.
Therefore, for every convex path $X$, $X \setminus S$ is a single path.

Let $L$ be the subgraph induced in $C$ by the set of vertices that are visited by Algorithm~\ref{alg:add}. $L$ is either a path or all of $C$.
An original reflex vertex of $C$ remains a reflex vertex of $C'$, unless it is added to $S$.
So Algorithm~\ref{alg:add} never passes over a reflex vertex of $C$ without adding it to $S$. 
Therefore, even if $S$ does not contain all the reflex vertices of $C$, there must exist a path in $L$ that 

\emph{(1)} contains every vertex in $S$, 

\emph{(2)} contains no reflex vertices in $C \setminus S$ and 

\emph{(3)} does not intersect the path $B \setminus S$ (if a maximal bad convex path $B$ exists). 

\noindent Since Algorithm~\ref{alg:add} begins inside $B$, $L$ is either a cycle, or does not cover all of $B$. Hence \emph{(3)} is satisfiable.
Let $A$ be the shortest path with these three properties.
By \emph{(2)} and \emph{(3)}, the components of $A \setminus S$ are good convex paths, so it is possible to construct pontoons over them.
Since $A$ is shortest possible, the end points of $A$ are in $S$.

Let $U$ be the longest path in $C \setminus S$. 
Thus $$|U| \geq \frac{2k^2-1-m}{m} \geq \frac{2k^2-1}{k-2} -1 > 2k.$$ 
We distinguish the following cases.

\begin{figure}
	\centering
	\includegraphics[scale=0.65]{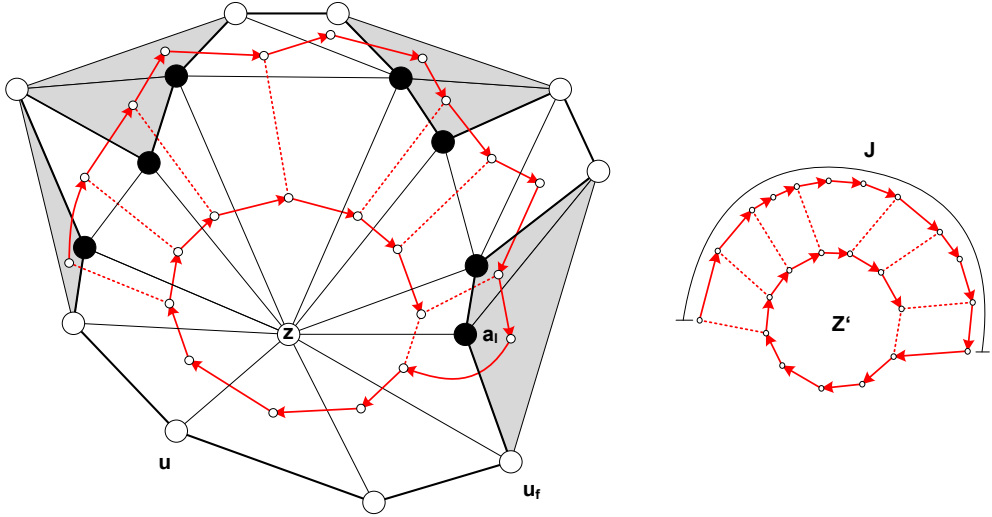}
	\caption{The graph of Figure~\ref{fig:AddingTriangles} after building all pontoons in Case~$1(a)$ (and a schematic drawing depicting $J$ and $Z'$). The thick edges in the weak dual depict the Hamiltonian path $H$ that is partitioned into blocks.}
	\label{fig:AddingTriangles2}
\end{figure}

\medskip
\emph{Case 1: All triangles added by Algorithm~\ref{alg:add} share an edge with $C$.}
\medskip

\noindent In this case, when Algorithm~\ref{alg:add} terminates, the weak dual consists of the cycle $Z$ along with $m$ vertices adjacent to $Z$ corresponding to the added triangles. 
These $m$ vertices induce a collection of paths.

\medskip
\noindent \emph{Case 1(a):} First suppose $U$ is not contained in $A$. 
Build a pontoon over each component of $A \setminus S$. 
Let $G$ be the triangulation after the pontoons are built, and let $G^*$ be its weak dual.
When building the pontoons, the cycle $Z$ shrinks to a smaller cycle called the \emph{inner cycle} $Z'$ in $G^*$.
The triangles of the pontoons connect the $m$ added triangles, so $G^* - Z'$ is a path $J$ 
(see Figure~\ref{fig:AddingTriangles2}). 
$J$ starts and ends at the vertices corresponding to the triangles added at the first and last vertices of $A$ respectively.

$G^*$ contains a Hamiltonian path $H$ that begins at the start of $J$, then traverses $J$ and passes from the last vertex of $J$ to the inner cycle $Z'$, visiting the rest of the vertices in clockwise order.
We call this a \emph{spiral path}.
Since $|G^*|= n-1+m \equiv_{k-2} 0$, the path $H$ can be partitioned into blocks consisting of paths each with $k-2$ vertices.

We now show that the subgraph of $G$ corresponding to this block partition of $G^*$ is $2$-connected, and hence a $k$-angulation of $P$ by Proposition~\ref{2con}. 
In the wheel triangulation, all vertices in the outer cycle $C$ are neighbours of $z$.
In the final triangulation $G$ there are two large cycles, the new outer cycle $C'$, and the cycle through the neighbours of $z$, call it $C_z$.
Within $A$, the addition of triangles over the vertices of $S$ removes them from $C'$, while pontoons are built over $A \setminus S$, removing these vertices from $C_z$.
The vertices in $P \setminus (A \cup \{z\})$ make up%
\footnote{Since the endpoints of $A$ are in $S$, $U$ is maximal in $C \setminus S$, and $U$ is not contained in $A$, it follows that $U$ is the complement of $A \cup \{z\}$.} %
$U$ and remain in both $C'$ and $C_z$.
Hence these two cycles intersect in the path $U$, which has order at least $2k$.

Now consider the subgraph $G'$ of $G$ corresponding to the block partition described above.
Since $U$ has order at least $2k$, there is at least one edge left from $z$ to a vertex $u$ in $U$.
There is also an edge from $z$ to the clockwise last point $a_l$ in $A$, where the last block in $H$ ends.
All edges of the cycles $C'$ and $C_z$ remain except for possibly the edge from $a_l$ to the first vertex $u_f$ in $U$.
Therefore, there is a $2$-connected spanning subgraph of $G'$ consisting of the union of $C'$, $C_z - a_lu_f$ and the path $a_lzu$.
This subgraph is a $\theta$-graph (that is, a graph consisting of two vertices joined by three internally disjoint paths).

\medskip

\noindent \emph{Case 1(b):} Now suppose $U$ is contained in $A$.
Then $U$ is a good convex path of order at least $2k$.
Building pontoons over all of $A \setminus S$ (including $U$) may remove too many vertices from $Z$, causing the spiral path method used in the previous case to fail to create a $2$-connected graph. 
Instead, first build pontoons over the components of $A \setminus (S \cup U)$. 
Then build a 
pontoon $R$ (the order of which will be determined later) over at most half of $U$ at the clockwise extreme.

The weak dual graph $G^*$ now consists of the inner cycle $Z'$, and the remaining outer section $J:= G^* - Z'$.
As before, the triangles of the pontoons form paths connecting the paths induced by the $m$ added triangles.
However this time $J$ consists of two separate paths because the pontoon $R$ does not cover all of $U$ and $C \setminus A$ is not empty. 
Let $J_1$ be the path containing $R$, and let $J_2$ be the second path.

\begin{figure}
	\centering
	\includegraphics[scale=0.65]{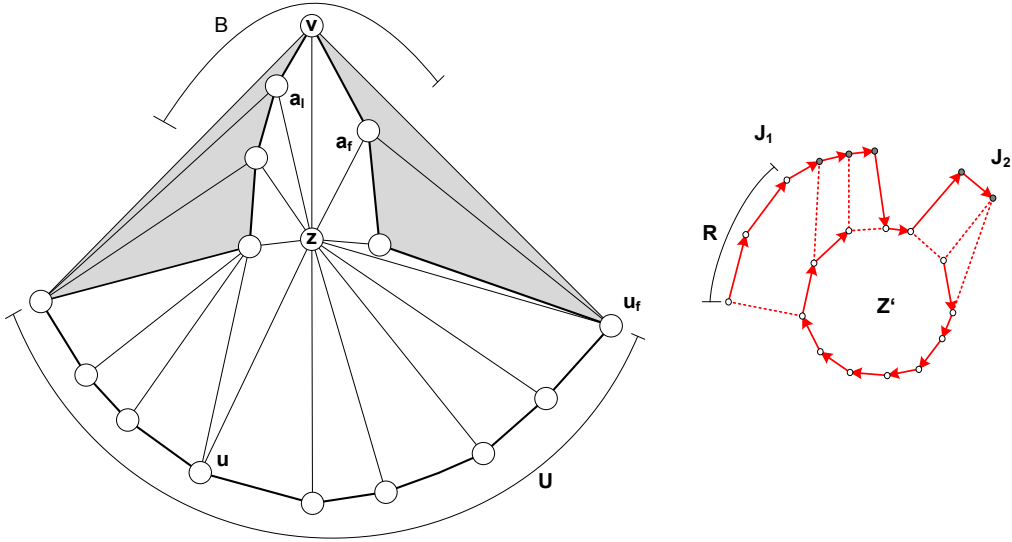}
	\caption{Case 1(b). The only vertex in $C\setminus A$ is $v$. The thick edges in the weak dual show the paths that are partitioned into blocks.}
	\label{fig:AddingTriangles4}
\end{figure}

We now describe two disjoint paths that together cover all the vertices of $G^*$.
The first path $H_1$ begins at the clockwise first vertex in $R$, traverses $J_1$ then passes to $Z'$, continues clockwise until it can pass to the clockwise first vertex of $J_2$ and then continues to the end of $J_2$.
Note that augmenting the pontoon $R$ by one triangle moves a vertex from $Z'$ to $J_1$ and therefore enlarges $H_1$ by one vertex. 
Since $|U| > 2k$, we can choose the order of $R$ such that the number of vertices in $H_1$ is a multiple of $k-2$ and still leave at least $k$ vertices in $Z'$. 
The remaining vertices form a second path $H_2$ that is contained in $Z'$. Since the total number of vertices in $G^*$ is a multiple of $k-2$, the order of $H_2$ must also be a multiple of $k-2$. Thus, both paths can be partitioned into blocks.

We again show that the subgraph $G'$ of $G$ corresponding to this block partition of $G^*$ is $2$-connected, and hence a $k$-angulation of $P$ by Proposition~\ref{2con}. 
The outer cycle $C'$ is a simple cycle.
The internal vertices all lie on the boundary between the blocks of $H_1$ and the blocks of $H_2$.
This boundary forms a path $D$ starting at the clockwise last vertex $u$ of $U$ that is not covered by the pontoon $R$, then proceeding clockwise through the neighbours of $z$ to the last vertex $a_l$ of $A$, then passing to $z$, then to the first vertex $a_f$ of $A$, and then proceeding clockwise through the neighbours of $z$ until the first vertex $u_f$ of $U$.
Thus $C' \cup D$ forms a spanning subgraph of $G'$ which is again a $\theta$-graph.

\medskip
\emph{Case 2: Some triangle added by Algorithm~\ref{alg:add} shares no edge with $C$.}
\medskip

\noindent Suppose a triangle added at a vertex $v \in C$ shares no edge with $C$. 
Then both the predecessor and successor of $v$ in $C$ must have had triangles added to them. 
This implies that Algorithm~\ref{alg:add} has visited every vertex of $C$, and so all reflex vertices of $C$ are in $S$. 
Therefore, all components of $C \setminus S$ (including $U$) are convex paths, one of which is not contained in $A$. 
By Claim 1 above, after Algorithm~\ref{alg:add} has terminated, $G_m^* - E(Z)$ is a collection of binary trees $T_1', \dots , T_{l'}'$ and the leaves of each tree are consecutive vertices on $Z$. 
Let $T_1, \dots , T_l$ be these trees minus their leaves, that is, the components of $G_m^* - V(Z)$.

Before giving the full details, we sketch the remainder of the proof.
We will begin by building pontoons between the trees $\{T_i\}$. 
Our goal is to partition the trees and pontoons into blocks separately from the inner cycle. 
Difficulties may arise from the fact that (unlike the paths in Case $1$) trees cannot in general be split into two components of arbitrary size. 
To avoid this problem, we will require that each tree is contained in a block (recall that no tree is larger than a block). 
This is possible if there are enough choices for where to start the first block.
To this end, we build an additional pontoon at the start of the outer section of the weak dual. The order of this pontoon is chosen so that no block starts inside a tree.

We now give a detailed account of this method.
\medskip

\noindent \emph{Case 2(a):} First suppose $U$ is not contained in $A$. 
As in Case~$1(a)$, build a pontoon over each component of $A \setminus S$. 
Even if $U$ is bad, it is possible to build a pontoon over at least half of its vertices, as either the first half or the last half must make an angle of less than $\pi$ at $z$, and hence form a good convex path. 
Since $|U| > 2k$, without loss of generality, it is possible to build a 
pontoon $R$ of order between $0$ and $k-3$ at the clockwise extreme of $U$. 
The precise order of $R$ will be determined later.
Let $G$ be the resulting triangulation once all pontoons are built.

The weak dual $G^*$ now consists of two sections, the inner cycle $Z'$ and the connected subgraph $J:= G^* - Z'$ consisting of the trees and the pontoons (see Figure~\ref{fig:AddingTriangles5}).
The vertices of the trees have an ordering given by the clockwise order of their corresponding vertices in $A$, and the vertices of the pontoons have a clockwise ordering in $J$.
Let $\sigma$ be a linear ordering of the vertices of $J$ that is consistent with these orderings, 
and where the intervals corresponding to each tree and pontoon are concatenated in the order they appear in $J$. 
That is, $\sigma$ begins with the vertices of $R$, then continues with the vertices of the first tree, then the next pontoon, then the second tree, and so on. 

\begin{figure}
	\centering
	\includegraphics[scale=0.65]{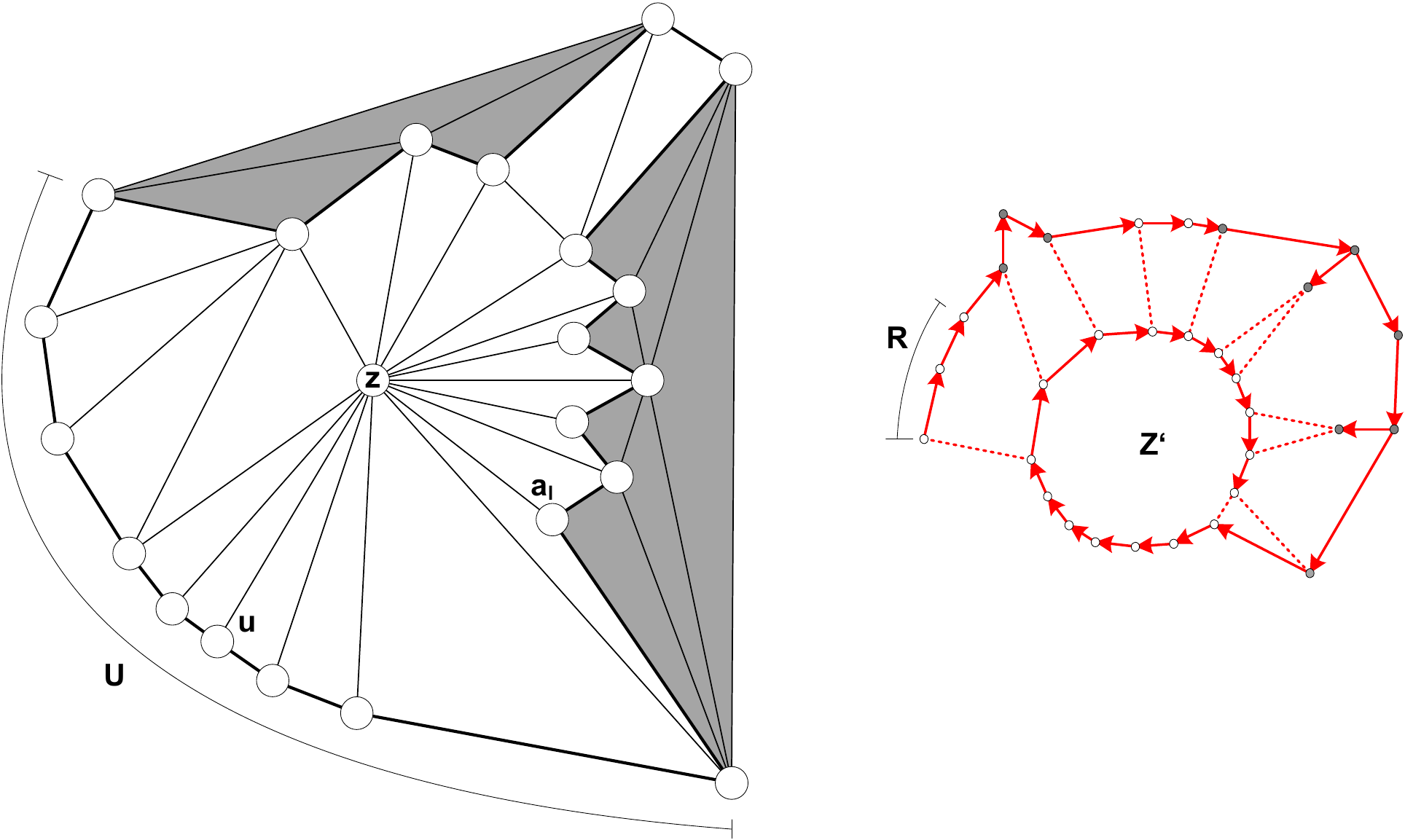}
	\caption{Case~2(a). $U$ is not contained in $A$. The partition is done in a spiral manner. $R$ prevents blocks from starting inside the trees.}
	\label{fig:AddingTriangles5}
\end{figure}

The partition of $G^*$ into blocks is done in a spiral manner similar to Case~$1(a)$, starting with $J$ then passing to $Z'$. $J$ is partitioned according to $\sigma$, taking the first $k-2$ vertices, then the next $k-2$ and so on.
The order of the initial pontoon $R$ may be between $0$ and $k-3$. 
Adjusting the order of $R$ shifts the blocks in $\sigma$, and each order determines a set of block starting points. These sets partition $\sigma \setminus R$. 
A block must not start inside a tree (though it may start at the first vertex of a tree), so there are less than $m$ places in $\sigma$ where a block may not start. 
Since $m < k-2$, there must be some set of block starting points that avoids these obstructions. 
This implies that there exists a feasible order for $R$ such that all trees lie inside a block. 

Just as in Case~$1(a)$, we now show that there is a $2$-connected spanning subgraph of the graph $G'$ corresponding to this block partition. 
Again there are two cycles $C'$ and $C_z$, but in this case their intersection is not the whole of $U$, but rather just the part of $U$ not covered by $R$. 
Nevertheless, this remainder has order greater than $k$, so there is again an edge left from $z$ to a vertex $u$ inside it.
There is also an edge from $z$ to $a_l$, the clockwise last vertex of $A$.
All edges of $C'$, and all but one of $C_z$ remain, and again there is a spanning $\theta$-subgraph consisting of these edges and the path $a_lzu$.
Hence, by Proposition~\ref{2con}, $G'$ is a $k$-angulation.

\medskip

\noindent \emph{Case 2(b):} Now suppose $U$ is contained in $A$. 
The goal is a partition of the weak dual into blocks similar to that in Case~$1(b)$.
First build pontoons over all components of $A \setminus (S \cup U)$. 
Next, at the clockwise extreme of $U$, build a 
pontoon $R$ of order at most $k-3$ which will play a similar role as in Case~$2(a)$, ensuring that blocks do not end inside the trees. 
Finally, at the counter-clockwise extreme of $U$, build a 
pontoon $L$ of order at most $k-3$ which will play a similar role as in Case~$1(b)$, ensuring that the last block of the outer section has exactly $k-2$ vertices. 
Now $G^*$ consists of the inner cycle $Z'$, the trees, and the pontoons (see Figure~\ref{fig:AddingTriangles6}). 
As in Case~$1(b)$, the outer section $J:= G^* -Z'$ has two components, $J_1$ containing $R$, and $J_2$ containing $L$.
As in Case~$2(a)$, define a linear ordering $\sigma$ of the vertices of $J$ starting at the first vertex of $R$.
This time $\sigma$ consists of two intervals, $\sigma_1$ covering $J_1$, and $\sigma_2$ covering $J_2$.

\begin{figure}
	\centering
	\includegraphics[scale=0.65]{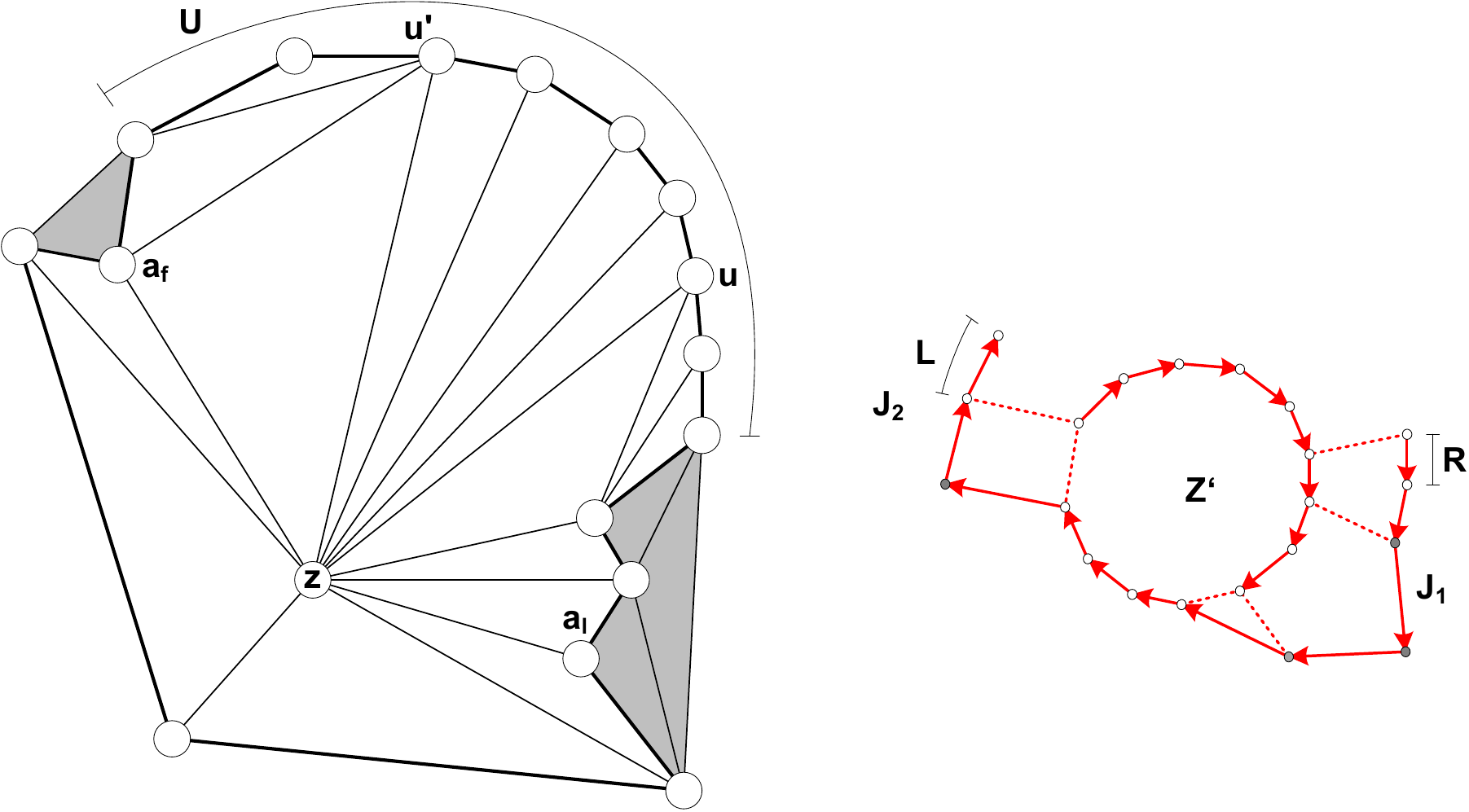}
	\caption{Case~2(b). $U$ is contained in $A$. The partition is done in two sections. $R$ prevents blocks from starting inside the trees, then $L$ finishes a block.}
	\label{fig:AddingTriangles6}
\end{figure}

The partition of $G^*$ into blocks has two sections (similar to the two paths in Case~$1(b)$). 
The first section starts at the first vertex of $R$, traverses $J_1$ according to the ordering $\sigma_1$, then passes to $Z'$, continuing clockwise until it can pass to the first vertex of $J_2$, then continues according to the ordering $\sigma_2$ to the last vertex of $L$. 
The order of $R$ is chosen so that no block ends inside a tree. 
Then the order of $L$ is chosen so that the total number of vertices in this section is divisible by $k-2$.
The second section of the partition is the remainder of $Z'$.
Since $|U| > 2k$, and building each of $L$ and $R$ removes at most $k-3$ vertices, there are at least $6$ vertices of $U$ not covered by $L$ and $R$. 
This implies that the remainder of $Z'$ is non-empty, and thus contains at least one block.

Just as in Case~$1(b)$, we now show that there is a $2$-connected spanning subgraph of the graph $G'$ corresponding to this block partition.
As before, it consists of the outer cycle $C'$ and a path $D$ covering the internal vertices.
$D$ starts at the clockwise last vertex $u$ of $U$ that is not covered by the pontoon $R$, then proceeds clockwise through the neighbours of $z$ to the last vertex $a_l$ of $A$, then passes to $z$, then to the first vertex $a_f$ of $A$, and then proceeds clockwise through the neighbours of $z$ until the first vertex $u'$ of $U$ that is not covered by the pontoon $L$.
Thus $C' \cup D$ forms a spanning $\theta$-subgraph of $G'$, so $G'$ is a $k$-angulation by Proposition~\ref{2con}. 
This concludes the proof of Theorem~\ref{mainthm}.
\end{proof}

\section{Remarks}

We conjecture that Theorem~\ref{mainthm} can be extended to hold for all $n \geq k$.

\begin{conjecture}
Let $n \geq k$ and $j \equiv_{k-2} k-n$ with $0 \leq j \leq k-3$.
A set $P$ of $n$ points in general position in the plane
admits a $k$-angulation if and only if it has at least $j$ interior points.
\end{conjecture}

\noindent It is not too hard to prove this conjecture for $k\leq6$ using an approach similar to that of this paper. 
%
%
Higher values of $k$ seem to require a significant modification of our methods.

\medskip

The proof of Theorem~\ref{mainthm} is constructive in the sense that it gives an efficient algorithm to compute $k$-angulations.

\begin{theorem}
Given an integer $k \geq 3$ and a set $P$ of $n \geq 2k^2$ points in general position in the plane, there is an $O(kn + n \log n)$ time algorithm that either computes a $k$-angulation on $P$ or reports that no such $k$-angulation exists.
\end{theorem}

\begin{proof}
The proof of Theorem~\ref{mainthm} describes the algorithm. Here we outline the complexity of the main steps. 
The first step is to compute the convex hull boundary of $P$, a task that can be done in $O(n \log n)$ time using standard methods (see for example \cite{deBergetal}). 
If there are less than $j$ interior points the algorithm terminates here and reports that no $k$-angulation exists (by Proposition~\ref{lem:interiorpoints}).

The case $j=0$ requires drawing an outerplanar $k$-angulation on $P$. 
If we select an outerplanar $k$-angulation whose dual is a path, Bose \cite{Bose02} gives an $O(n \log n)$ time algorithm to do this.
For $j\geq 1$, constructing the wheel triangulation, which requires sorting the points radially about $z$, can be done in $O(n \log n)$ time. 
The case $j=1$ is finished after the edges are removed from the initial triangulation to make a $k$-angulation; this takes $O(n)$ time.

For $j \geq 2$, we go through the major steps used in Algorithm~\ref{alg:add} and in the construction of the $k$-angulation. 
Whether an external vertex is reflex can be determined in $O(1)$ time. 
Thus, a bad convex path $B$ can be computed in $O(n)$ time.
In certain examples, Algorithm~\ref{alg:add} may find as few as one triangle to add per lap around $C$, so the number of steps is $O(kn)$. 
Having determined the set $S$, the paths $A$ and $U$, and thus the case distinction, can be calculated in $O(n)$ time.

Building a pontoon of order $p$ takes $O(p)$ time, so the pontoons joining parts of $S$ can be built in $O(n)$ time.
The dual graph can also be constructed with standard methods in $O(n)$ time. 
Computing the order of a pontoon such that the last block of the partition contains exactly $k-2$ vertices can be done in $O(n)$ time. 
In Cases~$2(a)$ and $2(b)$, we need to compute the order of a pontoon $R$ such that \emph{all} resulting blocks have valid starting points, that is, starting points that are not inside a tree. 
The order of $R$ may be between $0$ and $k-3$. 
In these cases, we simply iterate through each of the $k-2$ possibilities and check if all the resulting blocks have valid starting points. 
This takes $O(kn)$ time. 
The partition of the dual graph into blocks can be computed by traversing the dual graph in the manner described in $O(n)$ time. 
Finally, removing the edges to form a $k$-angulation takes $O(n)$ time.
In total, we get a running time of $O(kn + n \log n)$.
\end{proof}


\bibliographystyle{abbrv}
\bibliography{cites}

\begin{thebibliography}{10}

\bibitem{BiedlVatshelle}
T.~Biedl and M.~Vatshelle.
\newblock The point-set embeddability problem for plane graphs.
\newblock Tech. Report CS-2011-27, University of Waterloo, Canada, November
  2011.

\bibitem{Bose02}
P.~Bose.
\newblock On embedding an outer-planar graph in a point set.
\newblock {\em Comput. Geom. Theory Appl.}, 23(3):303--312, 2002.

\bibitem{Bose1995}
P.~Bose and G.~Toussaint.
\newblock No quadrangulation is extremely odd.
\newblock In {\em Algorithms and Computations}, volume 1004 of {\em Lecture
  Notes in Computer Science}, pages 372--381. Springer, 1995.

\bibitem{Bose1997}
P.~Bose and G.~Toussaint.
\newblock Characterizing and efficiently computing quadrangulations of planar
  point sets.
\newblock {\em Computer Aided Geometric Design}, 14(8):763--785, 1997.

\bibitem{Cabello}
S.~Cabello.
\newblock Planar embeddability of the vertices of a graph using a fixed point
  set is {NP}-hard.
\newblock {\em J. Graph Algorithms Appl.}, 10(2):353--363 (electronic), 2006.

\bibitem{CastUrrutia}
N.~Castañeda and J.~Urrutia.
\newblock Straight line embeddings of planar graphs on point sets.
\newblock In {\em Proc. Eighth Canadian Conference on Computational Geometry},
  pages 312--318, 1996.

\bibitem{deBergetal}
M.~de~Berg, O.~Cheong, M.~van Kreveld, and M.~Overmars.
\newblock {\em Computational geometry}.
\newblock Springer-Verlag, Berlin, third edition, 2008.
\newblock Algorithms and applications.

\bibitem{Dey1997}
T.~K. Dey, M.~B. Dillencourt, S.~K. Ghosh, and J.~M. Cahill.
\newblock Triangulating with high connectivity.
\newblock {\em Comput. Geom. Theory Appl.}, 8(1):39--56, 1997.

\bibitem{more3}
E.~Di~Giacomo, W.~Didimo, G.~Liotta, H.~Meijer, F.~Trotta, and S.~K. Wismath.
\newblock {$k$}-colored point-set embeddability of outerplanar graphs.
\newblock {\em J. Graph Algorithms Appl.}, 12(1):29--49, 2008.

\bibitem{more1}
E.~Di~Giacomo, W.~Didimo, G.~Liotta, H.~Meijer, and S.~K. Wismath.
\newblock Constrained point-set embeddability of planar graphs.
\newblock {\em Internat. J. Comput. Geom. Appl.}, 20(5):577--600, 2010.

\bibitem{more4}
E.~Di~Giacomo, W.~Didimo, G.~Liotta, and S.~K. Wismath.
\newblock Book embeddability of series-parallel digraphs.
\newblock {\em Algorithmica}, 45(4):531--547, 2006.

\bibitem{more2}
E.~Di~Giacomo, G.~Liotta, and F.~Trotta.
\newblock Drawing colored graphs with constrained vertex positions and few
  bends per edge.
\newblock {\em Algorithmica}, 57(4):796--818, 2010.

\bibitem{Garcia2009}
A.~Garc{\'\i}a, F.~Hurtado, C.~Huemer, J.~Tejel, and P.~Valtr.
\newblock On triconnected and cubic plane graphs on given point sets.
\newblock {\em Comput. Geom. Theory Appl.}, 42:913--922, 2009.

\bibitem{gritzmohar}
P.~Gritzmann, B.~Mohar, J.~Pach, and R.~Pollack.
\newblock Problems and {S}olutions: {S}olutions of {E}lementary {P}roblems:
  {E}3341.
\newblock {\em Amer. Math. Monthly}, 98(2):165--166, 1991.

\bibitem{SchmidtValtr}
J.~M. Schmidt and P.~Valtr.
\newblock Cubic plane graphs on a given point set.
\newblock In {\em Proceedings of the 28th Annual Symposium on Computational
  Geometry (SoCG'12)}, to appear (2012).

\bibitem{Thomassen1981}
C.~Thomassen.
\newblock Kuratowski's theorem.
\newblock {\em Journal of Graph Theory}, 5(3):225--241, 1981.

\bibitem{Toussaint1995}
G.~Toussaint.
\newblock Quadrangulations of planar sets.
\newblock In {\em Algorithms and Data Structures}, volume 955 of {\em Lecture
  Notes in Comput. Sci.}, pages 218--227. Springer, Berlin, 1995.

\end{thebibliography}

\end{document}